\documentclass[11pt]{amsart}
\usepackage{hyperref}
\usepackage{cmap,amsmath,amsthm,amssymb,amscd} 
\usepackage{enumerate}
\usepackage{amsbsy}
\usepackage{graphics}
\usepackage{graphicx}
\usepackage{verbatim} 
\usepackage{grffile}
\usepackage{times}

\usepackage{graphicx,subfigure,epsfig}
\usepackage{mathtools}
\usepackage[usenames,dvipsnames]{pstricks}

\DeclarePairedDelimiter\abs{\lvert}{\rvert}

\newtheorem{theorem}{Theorem}[section]
\newtheorem{corollary}[theorem]{Corollary}
\newtheorem{lemma}[theorem]{Lemma}

\newtheorem{remark}[theorem]{Remark}
\newtheorem{defi}[theorem]{Definition} 

\newtheorem{question}[theorem]{Question}

\begin{document}

\baselineskip 17pt 

\title[Genus one knot polynomials]{Polynomials of genus one prime knots \\ of complexity at most five}

\author{Maxim Ivanov}
\address{Laboratory of Topology and Dynamics, Novosibirsk State University, Novosibirsk}   \email{m.ivanov2@g.nsu.ru}

\author{Andrei Vesnin}
\address{Laboratory of Topology and Dynamics, Novosibirsk State University, Novosibirsk; Sobolev Institute of Mathematics of SB RAS, Novosibirsk; Tomsk State University, Tomsk}
\email{vesnin@math.nsc.ru} 

%\date{\today}


\begin{abstract}
Prime knots of genus one admitting diagram with at most five classical crossings were classified by Akimova and Matveev in 2014. In 2018 Kaur, Prabhakar and Vesnin introduced families of $L$-polynomials and $F$-polynomials for virtual knots which are generalizations of affine index polynomial. Here we introduce a notion of totally flat-trivial knots and demonstrate that for such knots $F$-polynomials and $L$-polynomials coincide with affine index polynomial. We prove that all Akimova~-- Matveev knots are totally flat-trivial and calculate their affine index polynomials.
\end{abstract}

\thanks{This work was supported by the Laboratory of Topology and Dynamics, Novosibirsk State University (contract no. 14.Y26.31.0025 with the Ministry of Science and Education of the Russian Federation).}

\maketitle 

\section*{Introduction}
Tabulating of virtual knots and constructing their invariants is one of the key problems in mordern low-dimensional topology. Table of virtual knots with diagrams, having at most four classical crossings may be found in monography ~\cite{dye} and online ~\cite{Green}. Due to equivalence of virtual knots and knots in thickened surfaces, it's interesting to consider tabulation of knots in 3-manifolds, which are thickenings of surfaces of certain genus. Up to now, there are just few results in this direction. Here we consider prime knots of genus one, admitting diagrams with small number of classical crossings, tabulated by Akimova and Matveev in ~\cite{akimova}. 

We are intrested in behaviour of several polynomial invariants on Akimova~-- Matveev knots. Recall that Kaufman in ~\cite{kauffman2013affine} defined an afiine index polynomial which is an invariant of a virtual knot and possess some important proprties~\cite{kauffman2018cobordism}. In ~\cite{KPV} a generalization of affine index polynomials was introduced, namely a family of $L$-polynomials $\{ L^n_K(t,\ell)\}_{n=1}^{\infty}$ and family of $F$-polynomials $\{ F^n_K(t,\ell)\}_{n=1}^{\infty}$. In ~\cite{IV} authors, using their software, calculated $F$-polynomials of knots tabulated in \cite{dye}  and ~\cite{Green}. Here we consider polynomial invariants for knots in a thickened torus.

The paper has the following structure: in Section~\ref{sec1} we recall some basic definitions and facts to use further, in Section~\ref{sec2}  we introduce totally flat-trivial knots and show that for these knots $L$-polynomials and $F$-polynomials coincide with affine index polynomial, in Section~\ref{sec3} we calculate these invariants for Akimova -- Matveev knots. In Theorem~\ref{theorem3-1} we show that Akimova~-- Matveev knots are totally flat-trivial. In Corollary~\ref{corollary-3.2} and Table~\ref{table-2} their affine index polynomials are given. The investigation of properties of Akimova~-- Matveev knots leads to the following Question~\ref{quest}: Is it true, that every virtual knot of genus one is totally flat-trivial?

\section{Basic definitions} \label{sec1} 

Virtual knots and links were introduced by Louis Kaufman in ~\cite{kauffman1999virtual} as an essential generalization of classical knots. Diagrams of virtual knots may have classical and virtual crossings both. Two virtual knots are \emph{equivalent} if and only if their diagrams could be transformed in each other by finite sequences of classical (RI, RII, RIII in Fig.~\ref{fig1a}) and virtual (VRI, VRII, VRIII and SV in Fig.~\ref{fig1b}) Reidemeister moves.

\medskip 

\begin{figure}[ht] 
\centering
\includegraphics[scale=0.44]{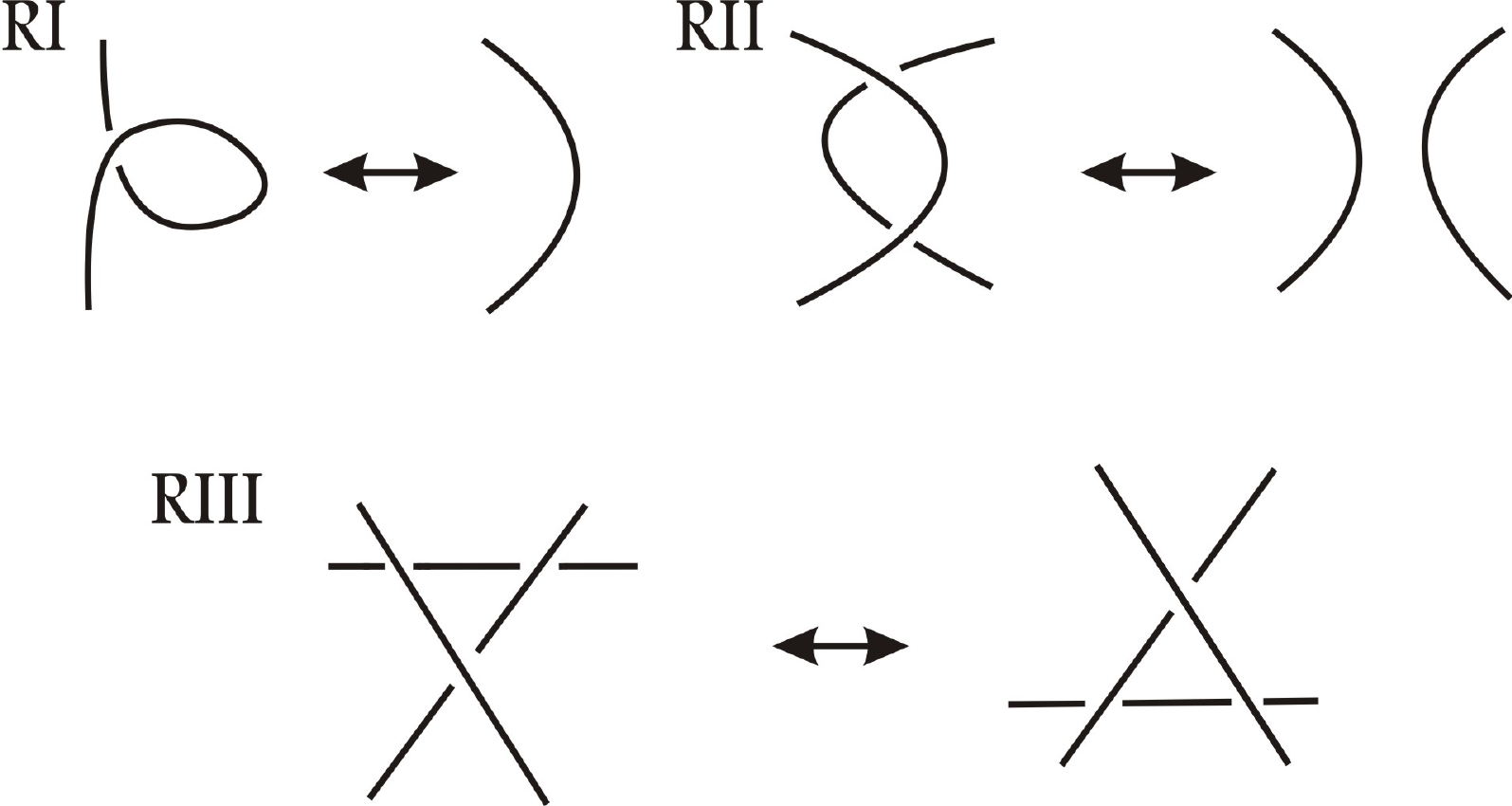} 
\caption{Classical Reidemeister moves.} \label{fig1a}
\end{figure} 

\begin{figure} 
\centering 
\includegraphics[scale=.44]{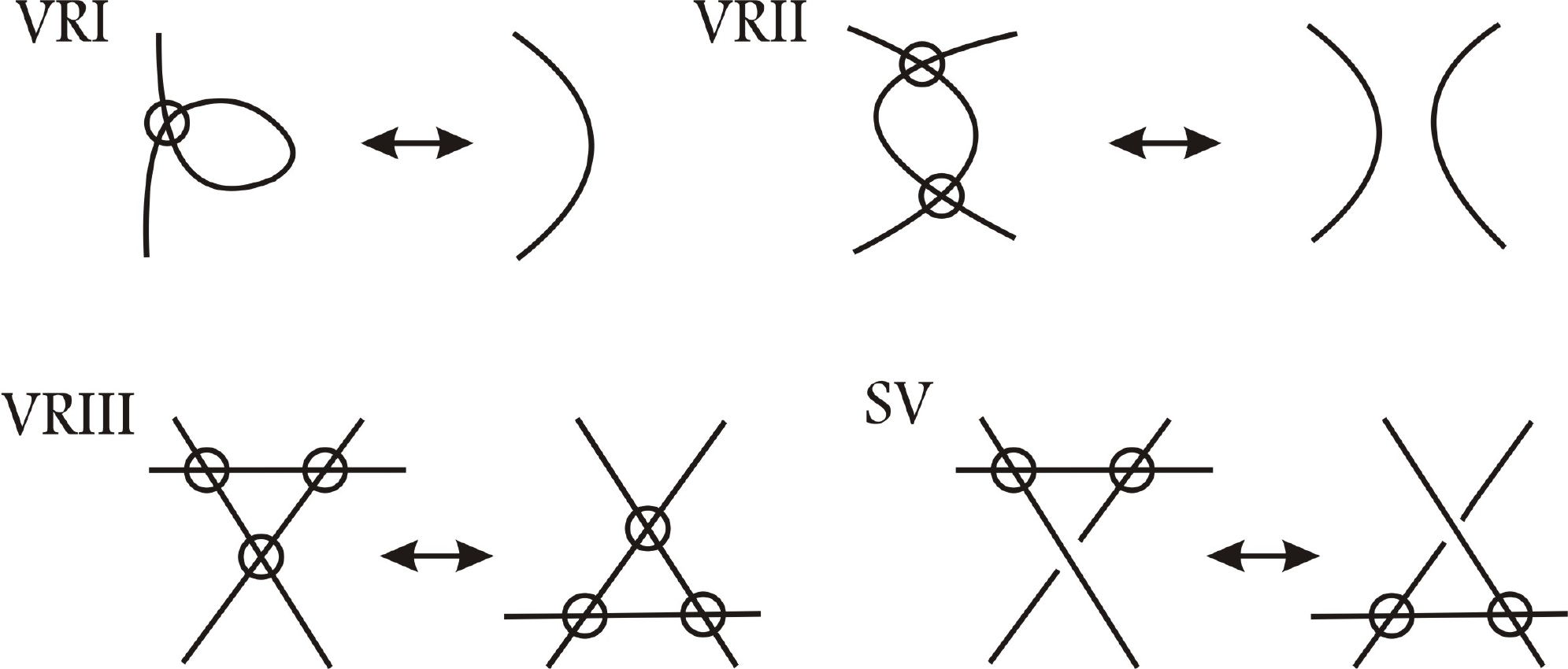}
\caption{Virtual Reidemeister moves.}   \label{fig1b}
\end{figure} 

Diagram, obtained by forgetting over/under crossing information is said to be  \emph{flat knot diagram}. Equivalence of flat knots is defined by \emph{flat Reidemeister moves}, which are different from \emph{virtual Reidemeister moves} in having flat crossings instead of classical ones.

Let $D$ be a diagram of an oriented virtual knot. We denote the set of all classical crossings of diagram $D$ as $C(D)$. Sign of a classical crossing, denoted by $\operatorname{sgn}(c)$ is defined as shown in the Fig.~\ref{fig102}.

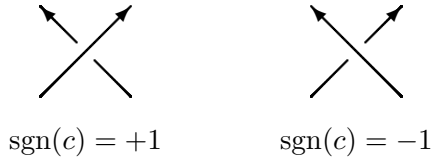
\begin{figure}[!ht]
\centering 
\unitlength=0.6mm
\begin{picture}(0,35)
\thicklines
\qbezier(-40,10)(-40,10)(-20,30)
\qbezier(-40,30)(-40,30)(-32,22) 
\qbezier(-20,10)(-20,10)(-28,18)
\put(-35,25){\vector(-1,1){5}}
\put(-25,25){\vector(1,1){5}}
\put(-30,0){\makebox(0,0)[cc]{$\operatorname{sgn}(c)=+1$}}
\qbezier(40,10)(40,10)(20,30)
\qbezier(40,30)(40,30)(32,22) 
\qbezier(20,10)(20,10)(28,18)
\put(25,25){\vector(-1,1){5}}
\put(35,25){\vector(1,1){5}}
\put(30,0){\makebox(0,0)[cc]{$\operatorname{sgn}(c)=-1$}}
\end{picture}
\caption{Signs of classical crossings.} \label{fig102}
\end{figure}

For every arc in a diagram of virtual knot we assign an integer value in such way that relations presented in a picture\ref{fig103} hold. In \cite{kauffman2013affine} Kaufman proved, that such coloring of an oriented virtual knot diagram, called \emph{Cheng coloring}, always exists. Indeed, for every arc $\alpha$ of a diagram $D$ one can assign value $\lambda (\alpha) = \sum_{c \in O(\alpha)} \operatorname{sgn} (c)$, where $O(\alpha)$ is the  set of classical crossings, which are fist met as overcrossings, when moving around the knot from $\alpha$ with respect to the orientation.

\begin{figure}[!ht]
\centering 
\unitlength=0.6mm
\begin{picture}(0,35)(0,5)
\thicklines
\qbezier(-70,10)(-70,10)(-50,30)
\qbezier(-70,30)(-70,30)(-62,22) 
\qbezier(-50,10)(-50,10)(-58,18)
\put(-65,25){\vector(-1,1){5}}
\put(-55,25){\vector(1,1){5}}
\put(-75,34){\makebox(0,0)[cc]{$b+1$}}
\put(-75,8){\makebox(0,0)[cc]{$a$}}
\put(-45,8){\makebox(0,0)[cc]{$b$}}
\put(-45,34){\makebox(0,0)[cc]{$a-1$}}
\qbezier(10,10)(10,10)(-10,30)
\qbezier(10,30)(10,30)(2,22) 
\qbezier(-10,10)(-10,10)(-2,18)
\put(-5,25){\vector(-1,1){5}}
\put(5,25){\vector(1,1){5}}
\put(-15,34){\makebox(0,0)[cc]{$b+1$}}
\put(-15,8){\makebox(0,0)[cc]{$a$}}
\put(15,8){\makebox(0,0)[cc]{$b$}}
\put(15,34){\makebox(0,0)[cc]{$a-1$}}
\qbezier(70,10)(70,10)(50,30)
\qbezier(70,30)(70,30)(50,10) 
\put(55,25){\vector(-1,1){5}}
\put(65,25){\vector(1,1){5}}
\put(60,20){\circle{4}}
\put(45,34){\makebox(0,0)[cc]{$b$}}
\put(45,8){\makebox(0,0)[cc]{$a$}}
\put(75,8){\makebox(0,0)[cc]{$b$}}
\put(75,34){\makebox(0,0)[cc]{$a$}}
\end{picture}
\caption{Cheng coloring} \label{fig103}
\end{figure}
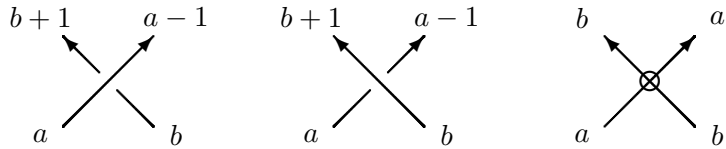
In \cite{cheng2013polynomial} Cheng and Gao put every classical crossing in correspondence with an integer value $\operatorname{Ind}(c)$, defined as
\begin{equation}
\operatorname{Ind}(c) = \operatorname{sgn} (c)(a-b-1)  \label{eq2.2}
\end{equation}
where $a$ and $b$  given by Cheng coloring. One can notice, that Cheng coloring does not depend on types of classical crossings and hence it is defined for an oriented flat knot diagram.
Let us remember that \emph{affine index polynomial} from~\cite{kauffman2013affine} can be written in the following form:
\begin{equation}
P_{D}(t) = \sum _{c\in C(D)} \operatorname{sgn}(c) (t^{\operatorname{Ind}(c)}-1),  \label{eq2}
\end{equation}
where $C(D)$ is a set of all classical crossings of $D$. 

In \cite{satoh2014writhes} Satoh and Taniguchi introduced a notion of $n$-writhe $J_n(D)$. For every $n \in \mathbb{Z} \setminus \{0\}$ define \emph{$n$-writhe} of oriented virtual knot diagram as a difference between number of positive crossings and negative crossings of index $n$.
Notice that $J_n (D)$ is a coefficient of $t^n$ in affine index polynomial and it is an invariant of oriented virtual knot. For more information about $n$-writhe see~\cite{satoh2014writhes}. Using $n$-writhe in \cite{KPV} was defined another invariant -- \emph{$n$-dwrithe} $\nabla J_{n}(D)$: 
$$ 
\nabla J_{n}(D)=J_{n}(D)-J_{-n}(D).
$$
%\smallskip 

\begin{remark} \label{rem2.1}
{\rm 
$\nabla J_{n}(D)$ is an invariant of oriented virtual knot, since $J_n(D)$ is an invariant of oriented virtual knot. Moreover, $\nabla J_n(D) = 0$ for every classical knot. 
}
\end{remark} 

As it shown in~\cite{KPV}, $\nabla J_n (D)$ represents a flat knot structure. Namely, the following lemma holds

\begin{lemma} \cite[Lemma~2.4]{KPV} \label{lemma1} 
For every $n \in \mathbb{N}$, $n$-dwrithe $\nabla J_n (D)$ is an oriented flat knot invariant. 
\end{lemma}

\smallskip 

Let $\bar{D}$ be a diagram, obtained from $D$ by reversing an orientation and $D^*$ is obtained by switching all classical crossings.    

\smallskip 

\begin{lemma} \cite[Lemma~2.5]{KPV} \label{lemma2}
Let $D$ be a diagram of oriented virtual knot, then $\nabla J_{n}(D^*)=\nabla J_{n}(D)$ and $\nabla J_{n}(\bar{D})=-\nabla J_{n}(D)$. 
\end{lemma} 

\smallskip 
Consider a smoothing according to the rule, shown in picture~\ref{fig105}. We will call this kind of smoothing by \emph{smoothing against orientation}.  Orientation of $D_{c}$ is induced by smoothing. Since $D$ is a diagram of virtual knot, so $D_{c}$ is a diagram of virtual knot too.  
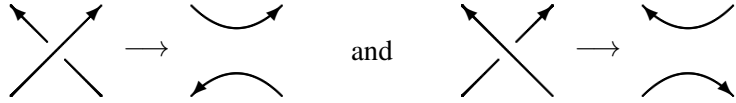
\begin{figure}[!ht]
\centering 
\unitlength=0.6mm
\begin{picture}(0,25)(0,10)
\thicklines
\qbezier(-80,10)(-80,10)(-60,30)
\qbezier(-80,30)(-80,30)(-72,22) 
\qbezier(-60,10)(-60,10)(-68,18)
\put(-75,25){\vector(-1,1){5}}
\put(-65,25){\vector(1,1){5}}
\put(-50,20){\makebox(0,0)[cc]{$\longrightarrow$}}
\qbezier(-40,30)(-30,20)(-20,30)
\qbezier(-40,10)(-30,20)(-20,10) 
\put(-35,15){\vector(-1,-1){5}}
\put(-25,25){\vector(1,1){5}}
\put(0,20){\makebox(0,0)[cc]{and}}
\qbezier(20,30)(20,30)(40,10)
\qbezier(20,10)(20,10)(28,18) 
\qbezier(40,30)(40,30)(32,22)
\put(25,25){\vector(-1,1){5}}
\put(35,25){\vector(1,1){5}}
\put(50,20){\makebox(0,0)[cc]{$\longrightarrow$}}
\qbezier(60,30)(70,20)(80,30)
\qbezier(60,10)(70,20)(80,10) 
\put(65,25){\vector(-1,1){5}}
\put(75,15){\vector(1,-1){5}}
\end{picture}
\caption{Smoothing.} \label{fig105}
\end{figure}

\smallskip 

\begin{defi} \cite{KPV} \label{l-pol}
{\rm 
For a diagram $D$ of a virtual oriented knot $K$ and an integer $n$, a polynomial $L_{K}^{n}(t, \ell)$ is defined as:
\begin{equation} 
L_{K}^{n}(t,\ell) = \sum_{c\in C(D)}\operatorname{sgn}(c) \left( t^{\operatorname{Ind}(c)}\ell^{\abs{\nabla J_{n}(D_{c})}}-\ell^{\abs{\nabla J_{n}(D)}} \right).  \label{eq3}
\end{equation} 
}
\end{defi} 

Note that $L$-polynomials generalize affine index polynomial, since $P_K(t) = L^n_K(t,1)$ for every $n$ and every $t$.

\begin{defi} \cite{KPV} \label{f-pol} 
{\rm For a diagram $D$ of a virtual oriented knot $K$ and an integer $n$, a polynomial $F_{K}^{n}(t, \ell)$ is defined as:
\begin{equation} 
\begin{gathered}
F_{K}^{n}(t,\ell) = \sum_{c \in C(D)} \operatorname{sgn}(c)t^{\text{Ind}(c)} \ell^{\nabla J_{n}(D_{c})}  
\qquad \qquad \qquad  \\  \qquad \qquad \qquad \qquad 
-  \sum _{c\in T_{n}(D)} \operatorname{sgn}(c) \ell^{\nabla J_{n}(D_{c})} - \sum _{c\notin T_{n}(D)} \operatorname{sgn} (c) \ell^{\nabla J_{n}(D)}, 
\end{gathered} \label{eq4} 
\end{equation} 
where  $T_{n}(D)=\{c \in C(D) :  | \nabla J_{n}(D_{c}) | \, =  \, | \nabla J_{n}(D) | \}$. 
}
\end{defi}

\begin{theorem} \cite{KPV} \label{lfth} 
For every integer $n\geq 1$ polynomials $L^{n}_{K}(t, \ell)$ and $F^{n}_{K}(t,\ell)$ are oriented virtual knot invariants.
\end{theorem}

\section{Totally  flat-trivial knots} \label{sec2} 

Let $D$ be a diagram of oriented virtual knot $K$ and $C(D)$ a set of all classical crossings in $D$.
\begin{defi} {\rm 
We will call $D$ \emph{totally flat-trivial} if diagrams obtained from $D$ and $D_{c}$ for all $c 
\in C(D)$ by forgetting over/under crossing information are flat equivalent to unknot. Virtual knot $K$ is said to be \emph{totally flat-trivial}, if it admits a totally flat-trivial diagram.
}
\end{defi}
 
 \begin{lemma}  \label{lemma3-2}
 If virtual knot $K$ is totally flat-trivial, then
 \begin{itemize}
\item[(1)] For all $n \geq 1$ we have $L^{n}_{K} (t, \ell)= P_{K} (t)$ and $F^{n}_{K} (t, \ell)= P_{K}(t)$.
\item[(2)] $P_{K}(t)$ is palindromic. 
\end{itemize} 
 \end{lemma}
 
 \begin{proof}
 (1) Let $D$ be a totally flat-trivial diagram of a knot $K$, $C(D)$ a set of all its classical crossings, and $D_{c}$ a diagram, obtained by smoothing against orientation in crossing $c \in C(D)$. By the definition, all these diagrams are flat-equivalent to a trivial knot. By Lemma~\ref{lemma1}  we have $\nabla J_{n}(D) =0$ and $\nabla J_{n} (D_{c}) =0$ for all $c \in C(D)$. From these equalities and formulas (\ref{eq2}), (\ref{eq3}), and (\ref{eq4}) we obtain  that $F^{n}_{K} (t, \ell) = L^{n}_{K} (t, \ell)= P_{K}$
 
 (2) It was mentioned above that $J_n (D)$ is a coefficient of $t^n$ in affine index polynomial. By the equality  $\nabla J_{n} (D) = J_{n} (D) - J_{-n} (D) = 0$, coefficients of $t^{n}$ and $t^{-n}$ coincide and $P_{K}(t)$ is palindromic.
 \end{proof}
 
Recall the following properties of affine index polynomial. Let $\bar{D}$ be a diagram, obtained from $D$ by reversing orientation and $D^*$ is obtained by switching all classical crossings.

\begin{lemma} \cite{kauffman2013affine} 
The following equalities hold
$$
P_{\bar{K}} (t) = P_{K} (t^{-1}) \qquad \text{and} \qquad P_{K^{*}} (t) = - P_{K} (t). 
$$
\end{lemma}

\smallskip 

\section{Polynomial invariants of Akimova-Matveev knots} \label{sec3} 

Prime knots in thickened torus $T \times I$, that is a product of 2-dimensional torus $T$ and the unit interval $I = [0,1]$, admitting diagrams with at most five classical crossings were tabulated by Akimova and Matveev in \cite{akimova}. The total number of these diagrams is equal to $90$. Due to Kuperberg's result~\cite{Kuper}, it is equivalent to tabulating prime virtual knots of genus one. To distinguish the knots, Akimova and Matveev introduced for diagrams on a torus an analogue of bracket polynomial.  These diagrams, pictured on a plane using virtual crossing are given in~\cite[Fig. 17]{akimova}. For reader's convenience we present them in Tables~\ref{table-3} and~\ref{table-4}.

\begin{theorem} \label{theorem3-1}
Every Akimova~--Matveev knot is totally flat-trivial.  
\end{theorem} 

\begin{proof} 
It's easy to see, that forgetting the information of over/under crossings in diagrams from Tables~\ref{table-3} and~\ref{table-4} leads us to 40 different diagrams of flat knots as in Table~\ref{table-1}. 
\begin{table}[h!]
\caption{Classes of diagrams.} \label{table-1} 
\begin{tabular}{|c|c||c|c||c|c||c|c||c|c|} \hline 
 & \text{knot} &  & \text{knot} &   & \text{knot} &  & \text{knot} &  & \text{knot}   \\ \hline 
1& 2.1 & 9&  4.10, 4.11 & 17&  5.10 & 25&  5.23, 5.24 & 33&  5.40-5.42 \\ \hline 
2&  3.1 & 10&  4.12-4.14 & 18&  5.11 & 26&  5.25, 5.26 & 34&  5.43-5.46 \\ \hline
3&  3.2, 3.3 & 11&  4.15-4.17 & 19&  5.12 & 27&  5.27-5.29 & 35&  5.47-5.50 \\ \hline 
4&  4.1& 12&  5.1, 5.2 & 20&  5.13 & 28&  5.30 & 36&  5.51-5.53 \\ \hline
5& 4.2& 13&  5.3, 5.4 & 21& 5.14 &  29&  5.31 &  37&  5.54-5.59 \\ \hline
6& 4.3 &  14&  5.5 & 22&  5.15, 5.16 & 30&  5.32, 5.33 & 38&  5.60-5.65 \\ \hline 
7&  4.4, 4.5 & 15&  5.6, 5.7 & 23&  5.17, 5.18 &  31&  5.34-5.37 &  39&  5.66-5.68 \\ \hline 
8&  4.6-4.9 & 16&  5.8, 5.9 &  24&  5.19-5.22 &  32&  5.38, 5.39 & 40& 5.69 \\ \hline
\end{tabular} 
\end{table}
Further we consider each of these classes separately and prove them to be totally flat-trivial. Changing type of a crossing leads to changing in orientation of a knot, obtained by smoothing at the crossing. Thereby it is sufficient to consider just one member from each of 40 classes to prove the theorem for the all 90 knots. 

As an example we consider a virtual knot $K = 5.17$ pictured in Fig.~\ref{fig5}.  

\smallskip 

\begin{figure}[h]
 \includegraphics[height=2.2cm]{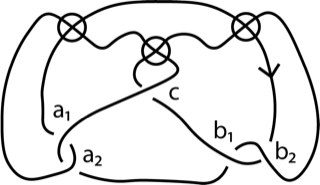}  
 \caption{Diagram of a virtual knot $K = 5.17$. } \label{fig5}
\end{figure}

It's easy to see, that it is flat-trivial.  
Then we consider all the diagrams obtained by smoothings in classical crossings. There are five classical crossings denoted as $a_{1}$, $a_{2}$, $b_1$, $b_{2}$, $c$.  Diagrams, obtained by smoothings at $a_{1}$, $b_{2}$ and $c$ are shown in the picture~\ref{fig6}.

\begin{figure}[h]
 \includegraphics[height=2.2cm]{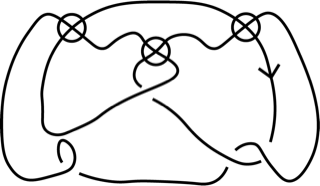}  \quad  \includegraphics[height=2.2cm]{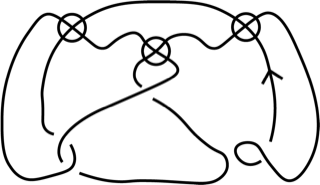} \quad  \includegraphics[height=2.2cm]{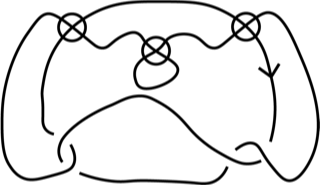}
 \caption{Diagrams, obtained by smoothings at crossings $a_{1}$, $b_{1}$ and $c$, respectively.} \label{fig6}
\end{figure} 

As one can see, all diagrams in Fig.~\ref{fig6} are also flat-trivial. Similarly, diagrams obtained by smoothing at $a_{2}$ and $b_{2}$ are also flat-trivial. Hence, virtual knot $K=5.17$ is totally flat-trivial.
Analogous considerations for knots from other classes show that they are all totally flat-trivial, and thus all Akimova-Matveev knots are totally  flat-trivial.
\end{proof}

Theorem~\ref{theorem3-1} and Lemma~\ref{lemma3-2} allow us o obtain the following properties of $L$-poly\-no\-mials, $F$-polynomials and affine index polynomial of Akimova~-- Matveev knots.

\begin{corollary} \label{corollary-3.2} 
Let $K$ be a genus one knot admitting a diagram with at most five crossings. Then for every $n\geq 1$ its $L$ polynomials and $F$-polynomials coincide with affine index polynomial, presented in Table~\ref{table-2}, where knots are splitted in groups with respect to the value of polynomials for the knot $K$ or its mirror image~$K^{*}$. 
\end{corollary}

\begin{table}[h!]
\caption{Polynomial invariants.} \label{table-2} 
\begin{tabular}{|c|c|} \hline 
\text{knot} $K$ & \mbox{polynomial} $P_{K}(t)$ \\ \hline 
4.4, 4.5, 5.15, 5.16, 5.27, 5.28, 5.29, 5.30,  & \\ 
5.31, 5.45, 5.47, 5.48, 5.67, 5.69  & $0$ \\  \hline
2.1, 3.1, 4.2, 5.6, 5.7*, 5.10*,  5.13*, 5.19, & \\ 
5.20, 5.21*, 5.22, 5.23, 5.24*, 5.43, 5.46   & $t^{-1}-2+t$ \\ \hline 
4.1*, 4.3, 5.5, 5.12, 5.44 & $2t^{-1}-4+2t$ \\ \hline 
3.2, 3.3, 4.6, 4.9, 4.10, 4.11, 5.3, 5.4*, 5.8,  5.9*, & \\ 
5.14, 5.17, 5.18*, 5.32, 5.33, 5.34, 5.37, 5.49, 5.66 & $t^{-2}-2+t^{2}$ \\ \hline 
5.1, 5.2, 5.11, 5.25, 5.26, 5.50, 5.68  & $2t^{-2}-4+2t^{2}$ \\ \hline
4.8, 5.35*, 5.39* & $t^{-2}-t^{-1}-t+t^{2}$ \\ \hline
4.7, 5.36, 5.38  & $t^{-2}+t^{-1}-4+t+t^{2}$ \\ \hline 
4.13, 4.15, 4.16, 4.17, 5.40, 5.41, 5.42*, 5.52, 5.53*  & $t^{-3}-2+t^{3}$ \\ \hline 
4.14 & $t^{-3}-t^{-1}-t+t^{3}$ \\   \hline 
4.12, 5.51 & $t^{-3}+t^{-1}-4+t+t^{3}$ \\ \hline 
5.54, 5.57, 5.60, 5.61, 5.62, 5.63, 5.64*, 5.65 & $t^{-4}-2+t^{4}$ \\ \hline
5.56, 5.68* & $t^{-4}-t^{-2}-t^{2}+t^{4}$ \\ \hline
5.55, 5.59 & $t^{-4}+t^{-2}-4+t^{2}+t^{4}$ \\ \hline
\end{tabular}
\end{table}
%%%%%%%%%%%

\begin{question} \label{quest} 
Is it true, that every virtual knot of genus one is totally flat trivial?
\end{question}

 %%%%%%%%%%%%%%%%%%%%%%%%%%%%%%%%%%%%%%%%%%%%%%%%%%%%%
\begin{table}[h!]
\caption{Diagrams of Akimova~-- Matveev knots (I).} \label{table-3}
\begin{tabular}{|c|c|c|c| } \hline
 \includegraphics[height=1.5cm]{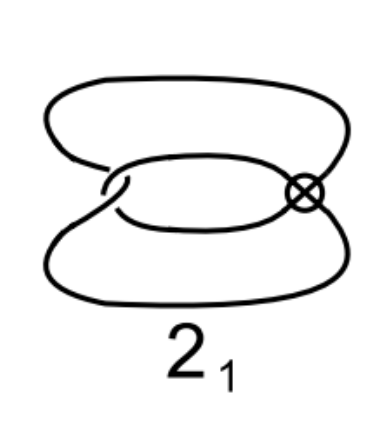} & 
 \includegraphics[height=1.5cm]{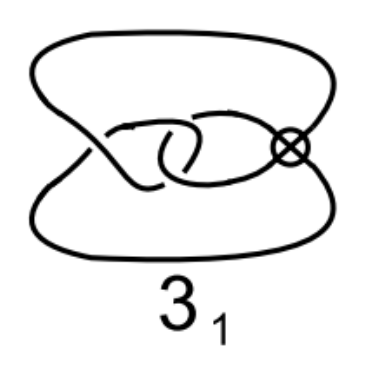} & 
\includegraphics[height=1.5cm]{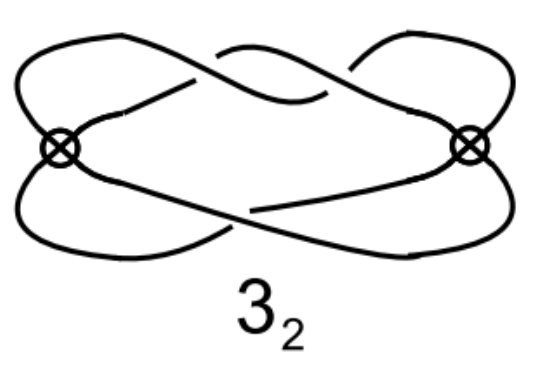} &
\includegraphics[height=1.5cm]{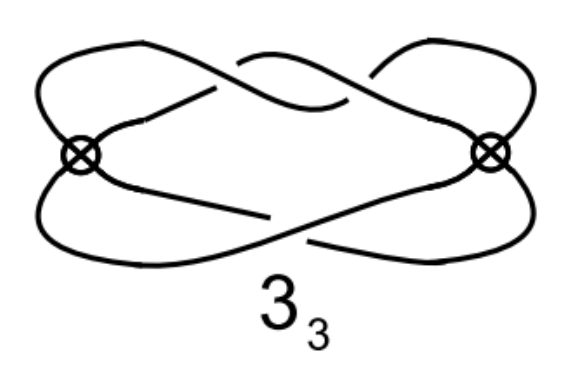} \\ \hline 
\includegraphics[height=1.5cm]{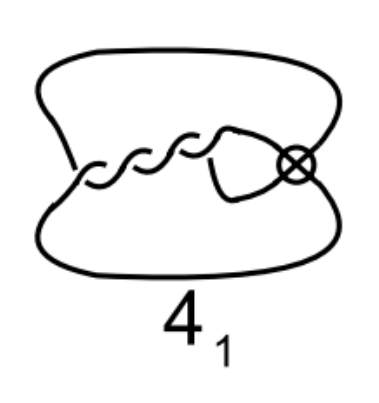}  & 
\includegraphics[height=1.5cm]{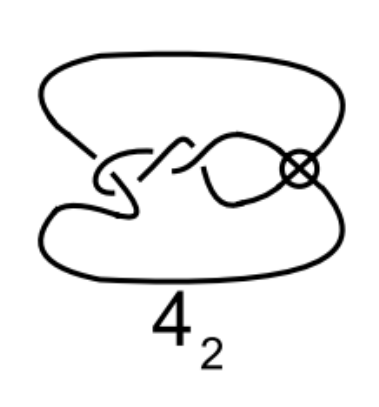} &
\includegraphics[height=1.5cm]{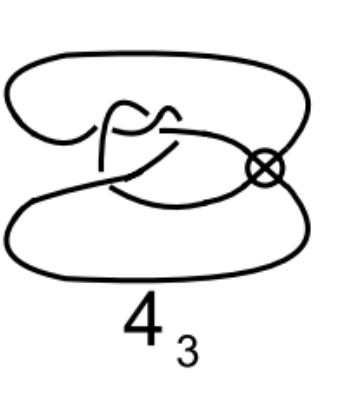} &
\includegraphics[height=1.5cm]{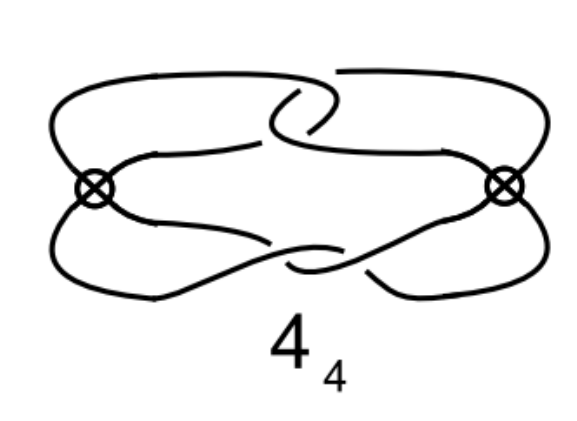}  \\ \hline 
\includegraphics[height=1.5cm]{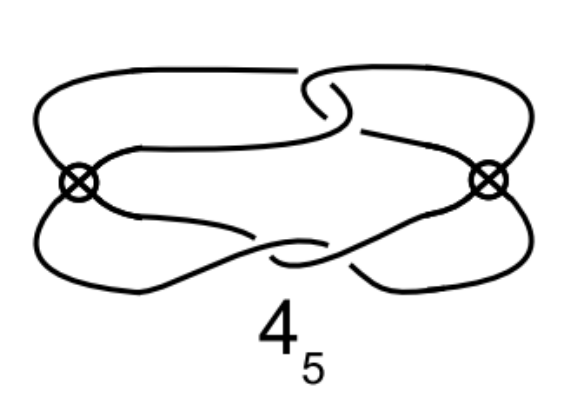} &
\includegraphics[height=1.5cm]{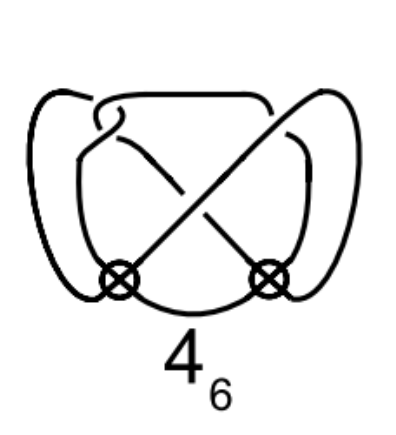} & 
\includegraphics[height=1.5cm]{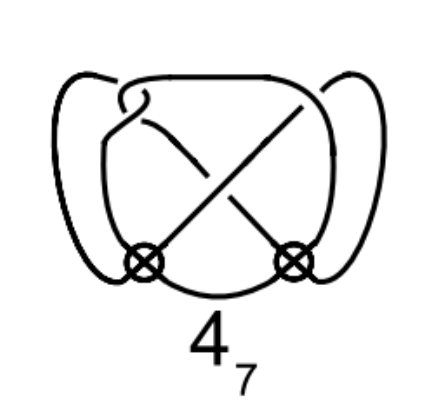} &
\includegraphics[height=1.5cm]{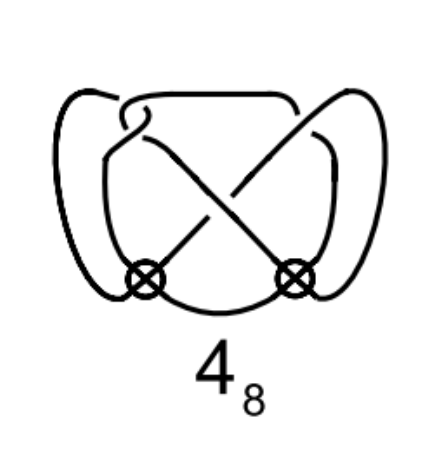} \\ \hline
\includegraphics[height=1.5cm]{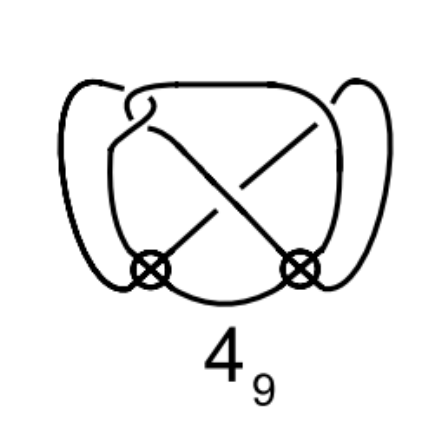} &
\includegraphics[height=1.5cm]{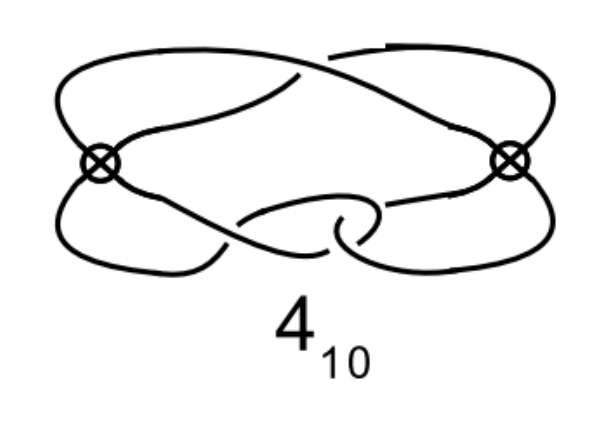} &
\includegraphics[height=1.5cm]{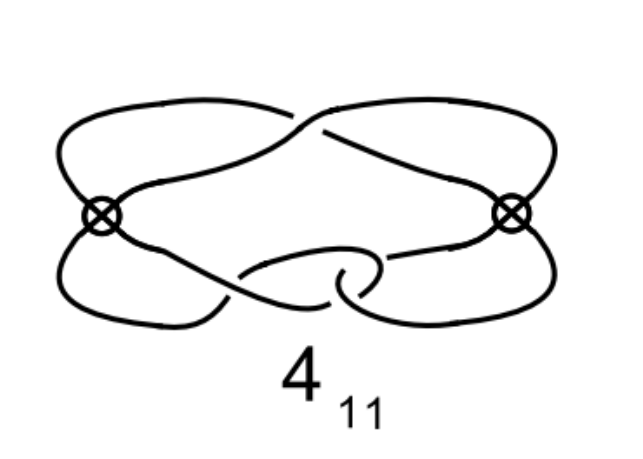} & 
\includegraphics[height=1.5cm]{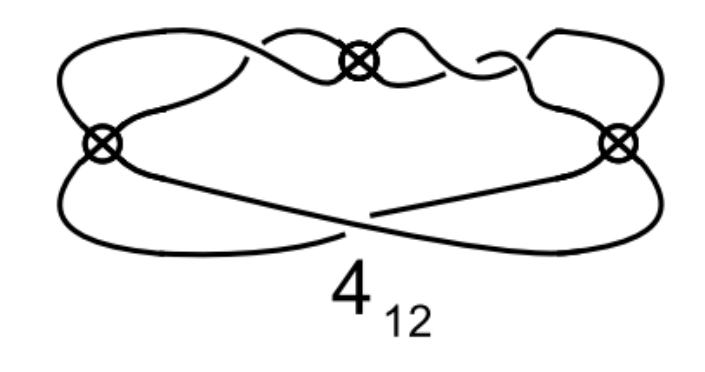} \\ \hline
\includegraphics[height=1.5cm]{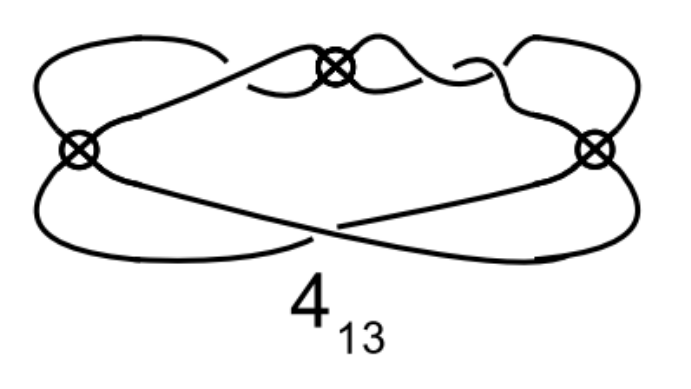} &
\includegraphics[height=1.5cm]{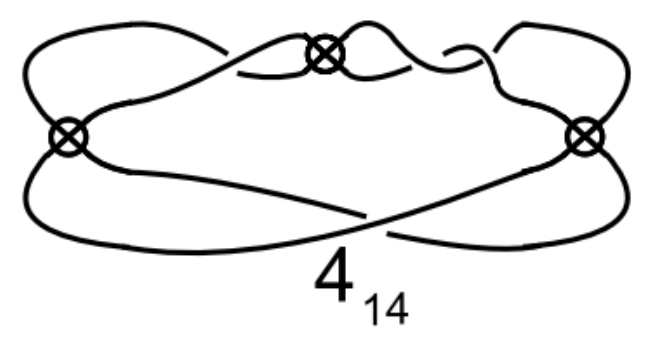} &
\includegraphics[height=1.5cm]{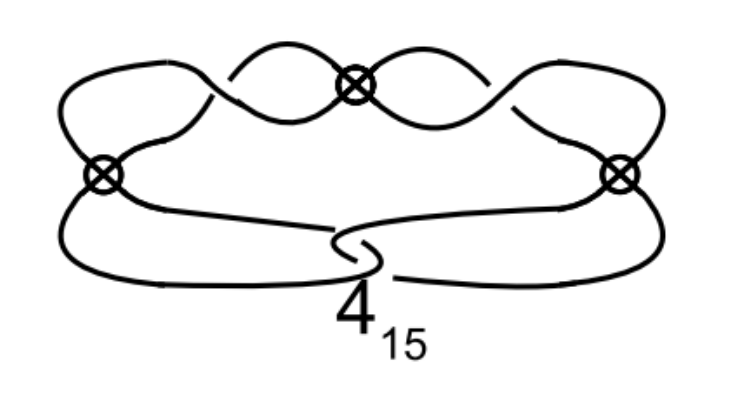} &
\includegraphics[height=1.5cm]{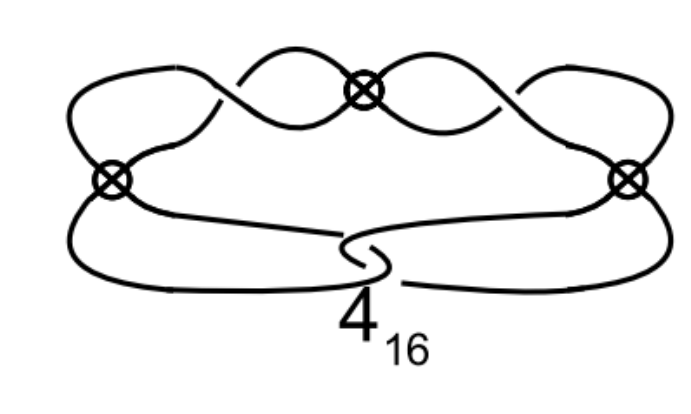} \\  \hline
\includegraphics[height=1.5cm]{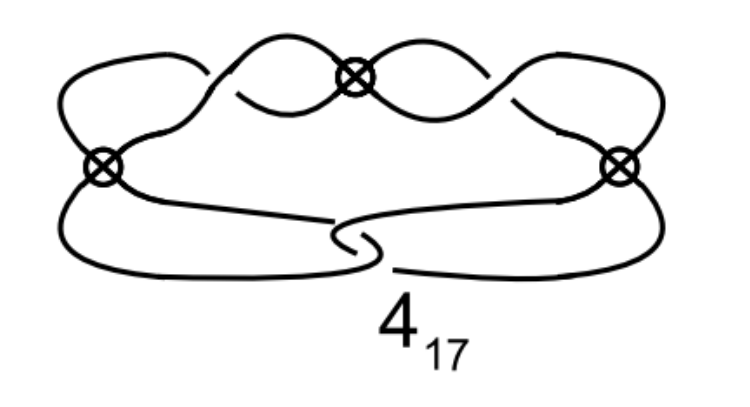} &
\includegraphics[height=1.5cm]{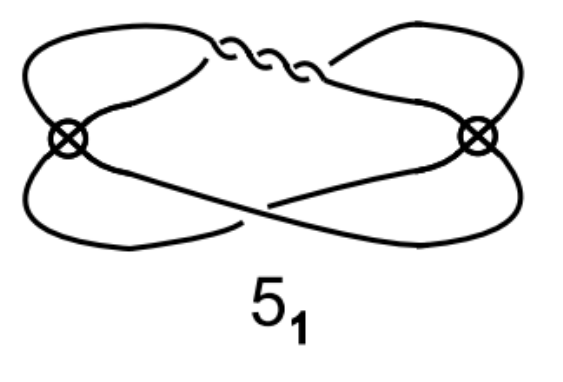} &
\includegraphics[height=1.5cm]{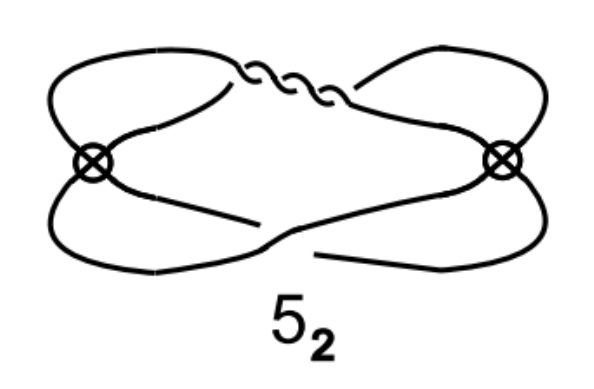} &
\includegraphics[height=1.5cm]{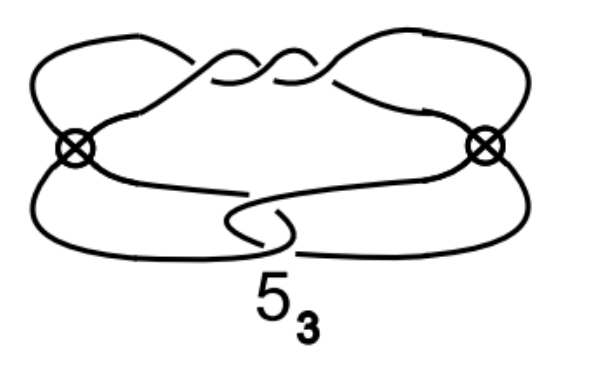} \\ \hline
\includegraphics[height=1.5cm]{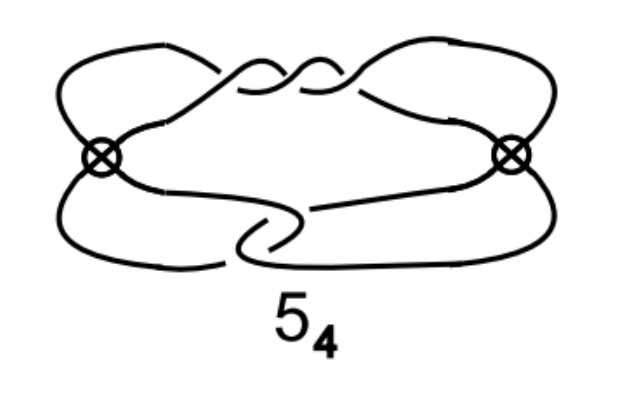} & 
\includegraphics[height=1.5cm]{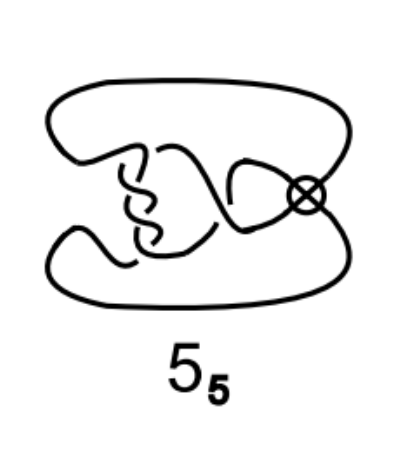} &
\includegraphics[height=1.5cm]{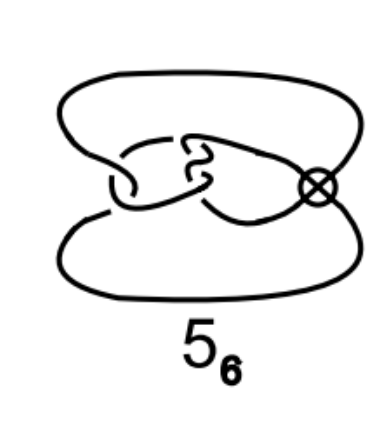} &
\includegraphics[height=1.5cm]{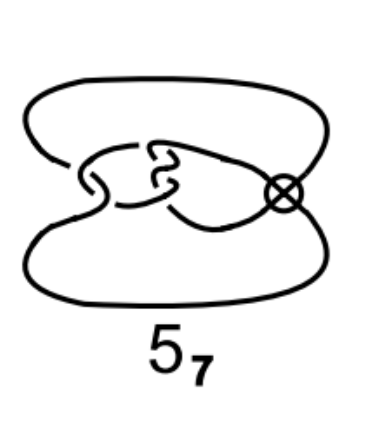} \\ \hline
\includegraphics[height=1.5cm]{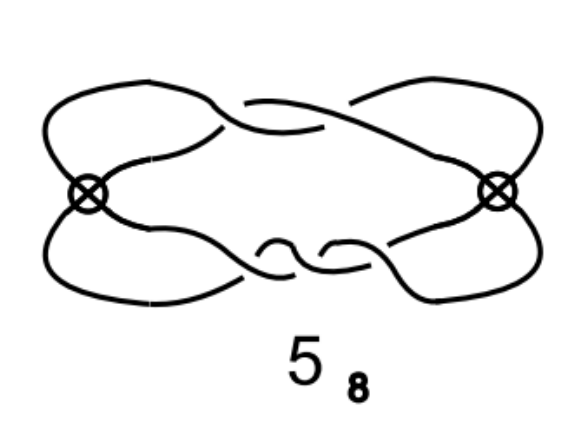} &
\includegraphics[height=1.5cm]{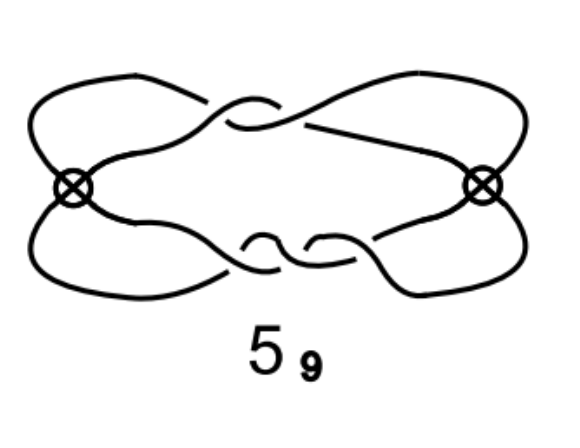} &
\includegraphics[height=1.5cm]{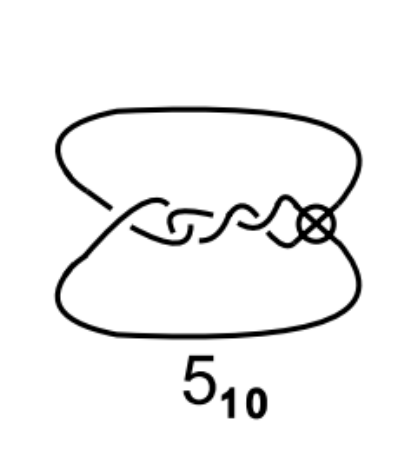} &
\includegraphics[height=1.5cm]{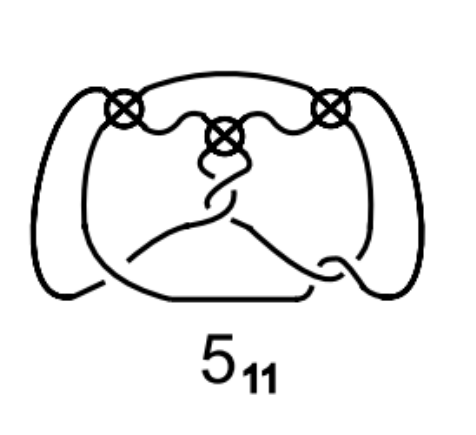} \\ \hline
\includegraphics[height=1.5cm]{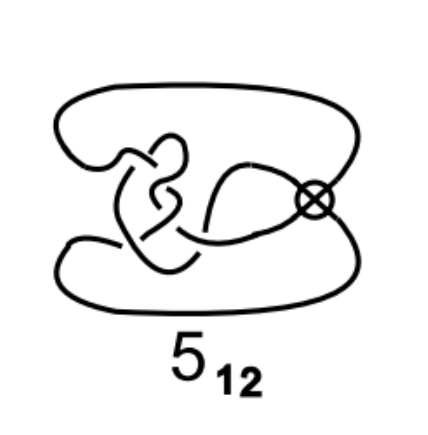} &
\includegraphics[height=1.5cm]{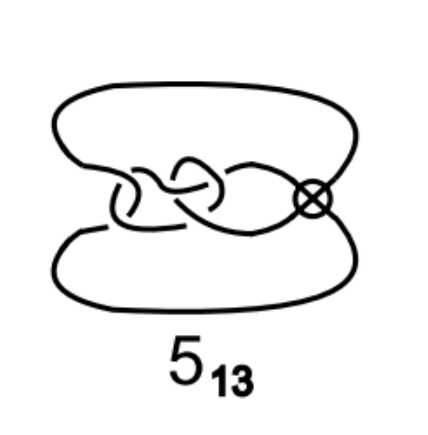} &
\includegraphics[height=1.5cm]{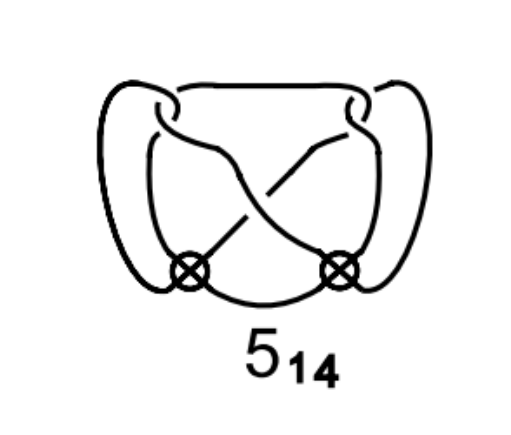} &  
\includegraphics[height=1.5cm]{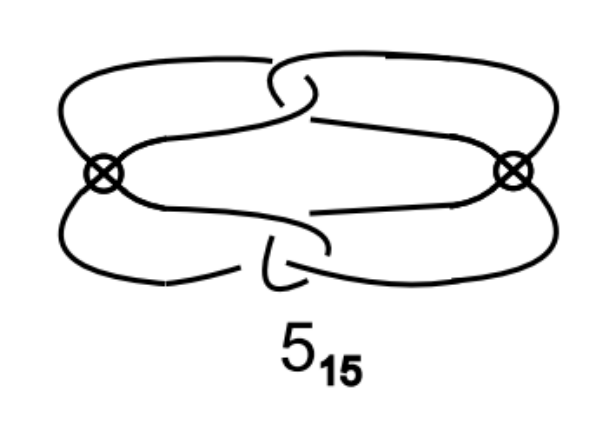} \\ \hline 
\includegraphics[height=1.5cm]{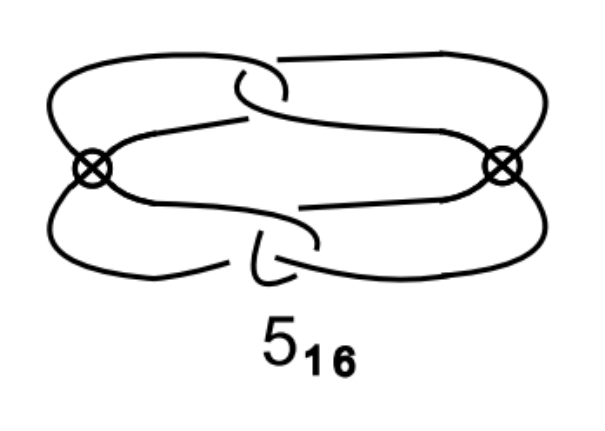} &
\includegraphics[height=1.5cm]{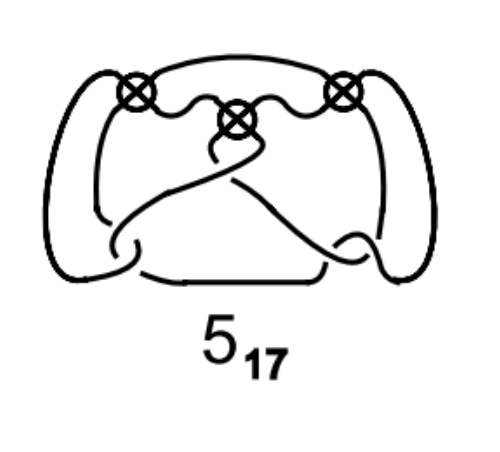} &
\includegraphics[height=1.5cm]{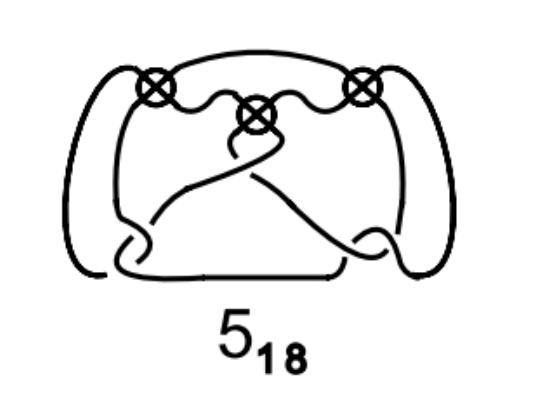} &
\includegraphics[height=1.5cm]{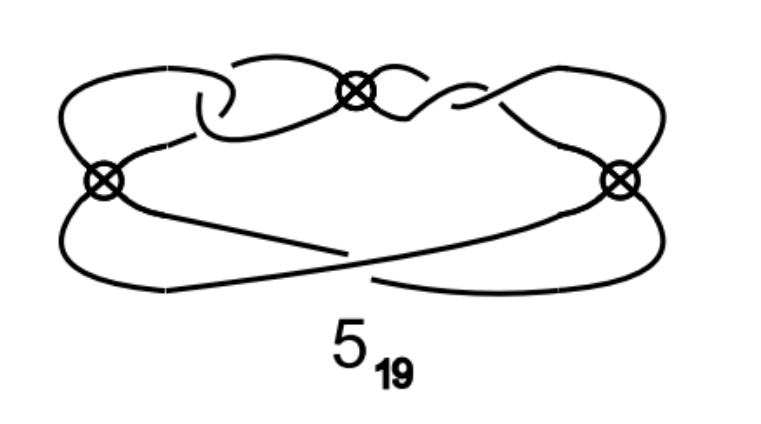} \\ \hline 
\includegraphics[height=1.5cm]{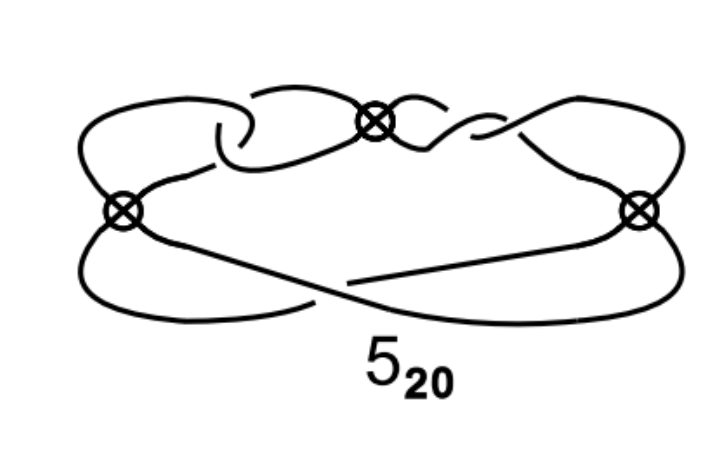} &
\includegraphics[height=1.5cm]{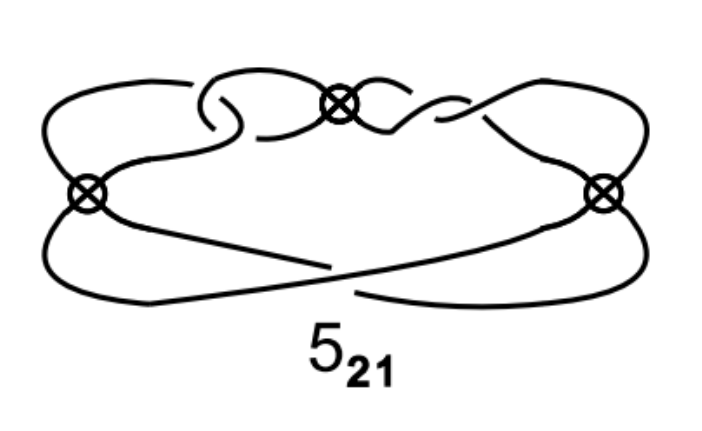} &
\includegraphics[height=1.5cm]{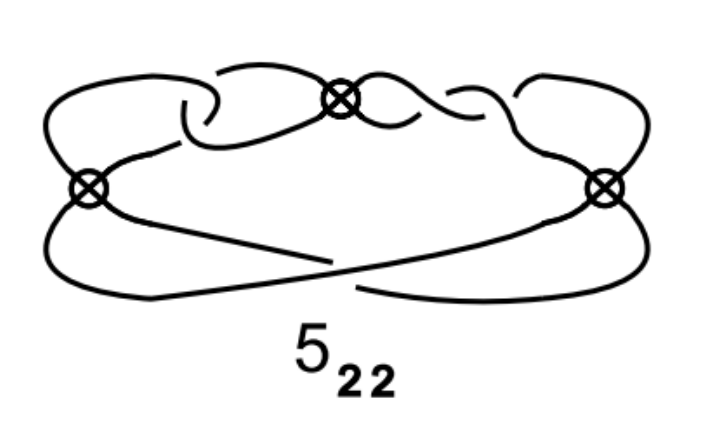} &
\includegraphics[height=1.5cm]{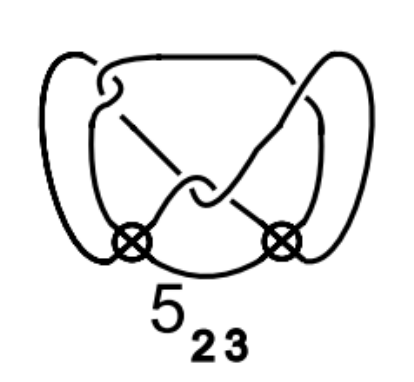} \\ \hline 
\includegraphics[height=1.5cm]{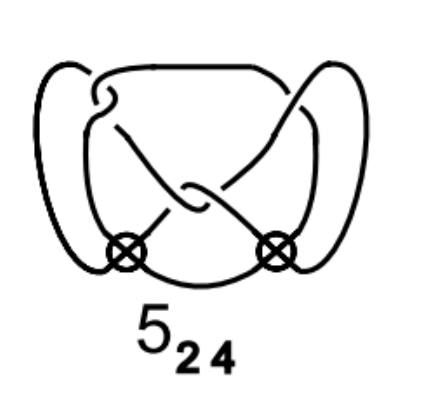} & 
 \includegraphics[height=1.5cm]{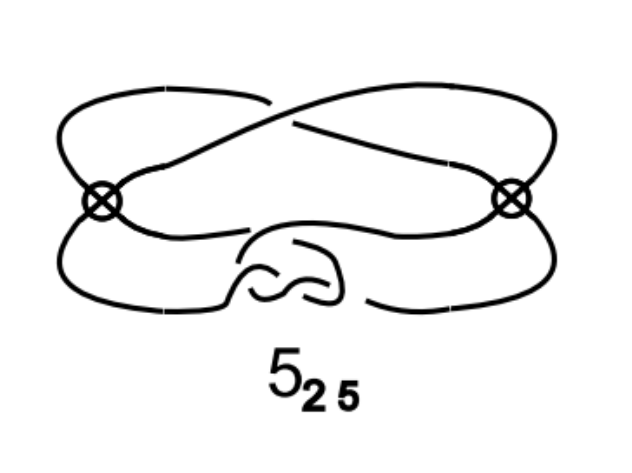} & 
 \includegraphics[height=1.5cm]{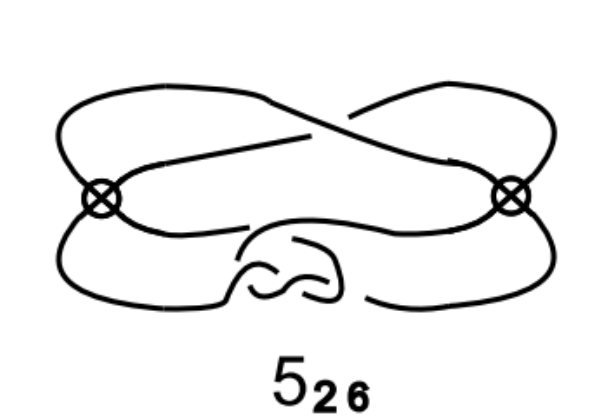} & 
\includegraphics[height=1.5cm]{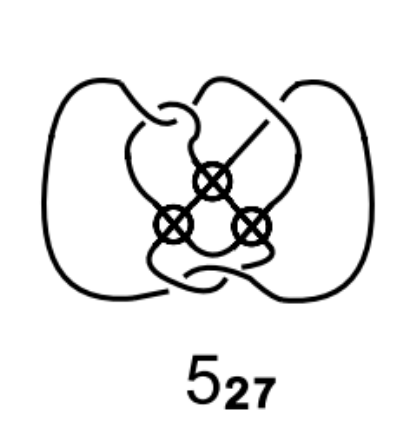} \\ \hline 
\includegraphics[height=1.5cm]{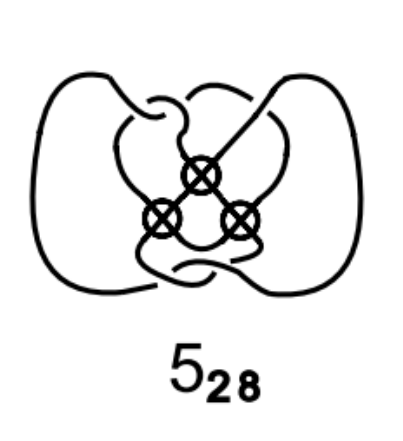} & 
\includegraphics[height=1.5cm]{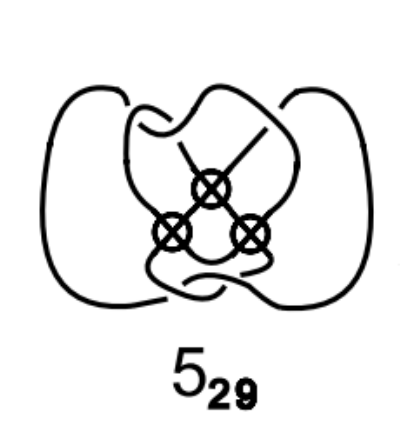}  & 
\includegraphics[height=1.5cm]{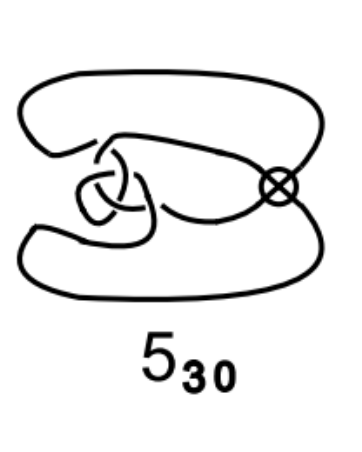} &
\includegraphics[height=1.5cm]{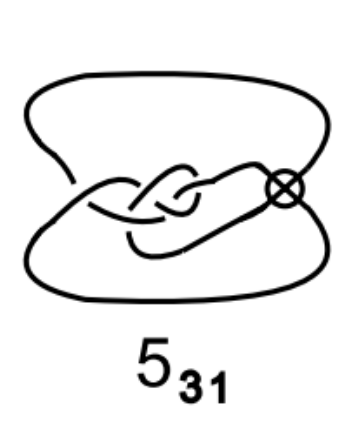} \\ \hline
\end{tabular}
\end{table}

\begin{table}[h!]
\caption{Diagrams of Akimova~-- Matveev knots (II).} \label{table-4}
\begin{tabular}{|c|c|c|c| } \hline
\includegraphics[height=1.5cm]{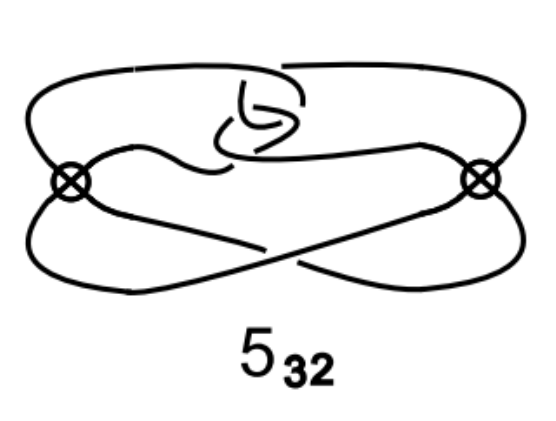} &
\includegraphics[height=1.5cm]{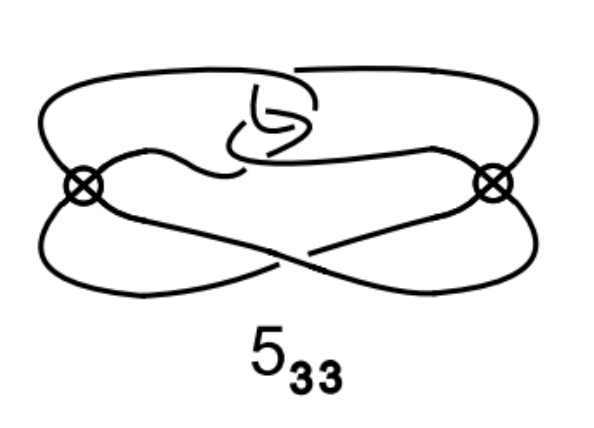} &
\includegraphics[height=1.5cm]{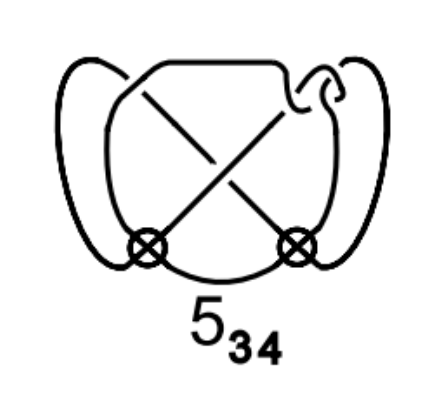} & 
\includegraphics[height=1.5cm]{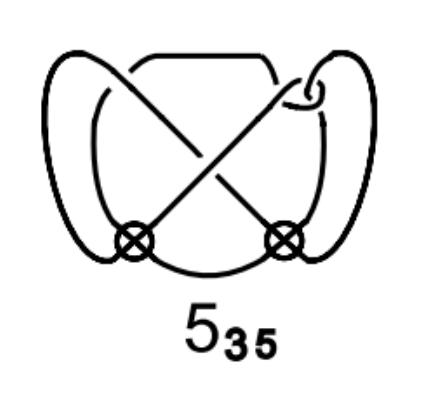} \\ \hline 
\includegraphics[height=1.5cm]{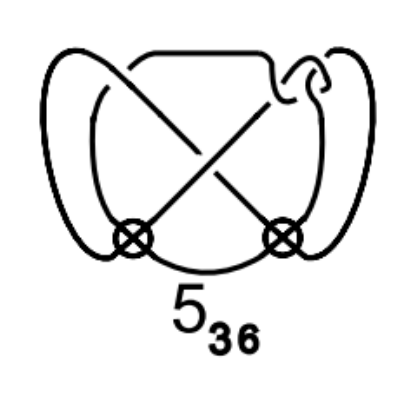} &
\includegraphics[height=1.5cm]{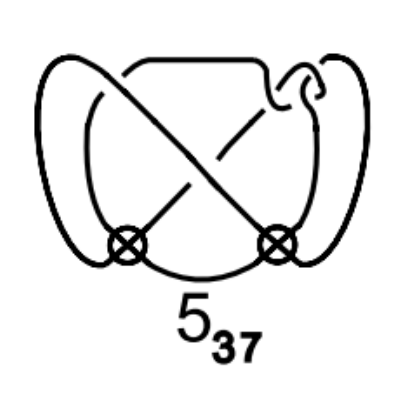} &
\includegraphics[height=1.5cm]{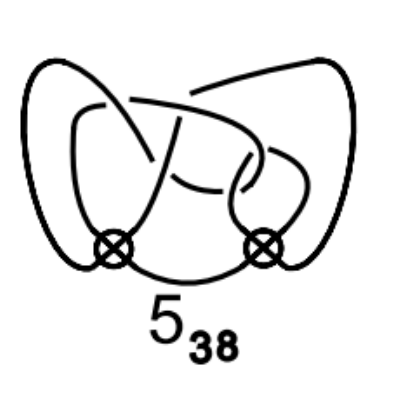} &
\includegraphics[height=1.5cm]{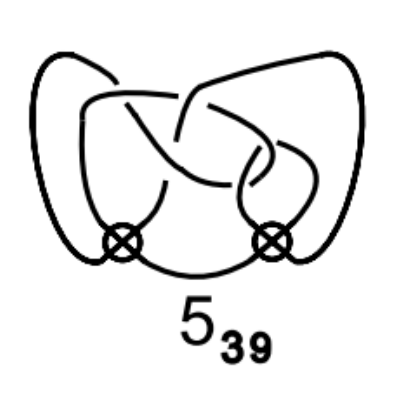} \\  \hline
\includegraphics[height=1.5cm]{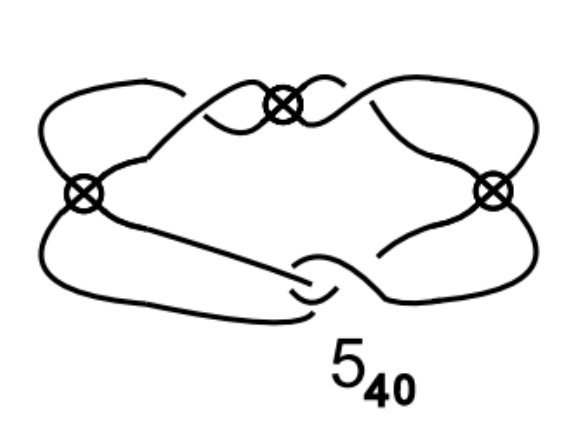} &
\includegraphics[height=1.5cm]{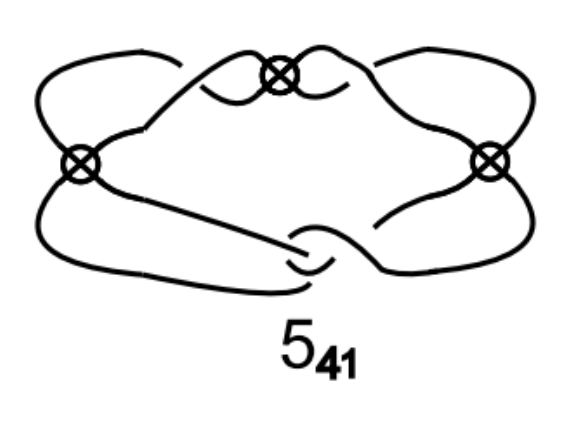} &
\includegraphics[height=1.5cm]{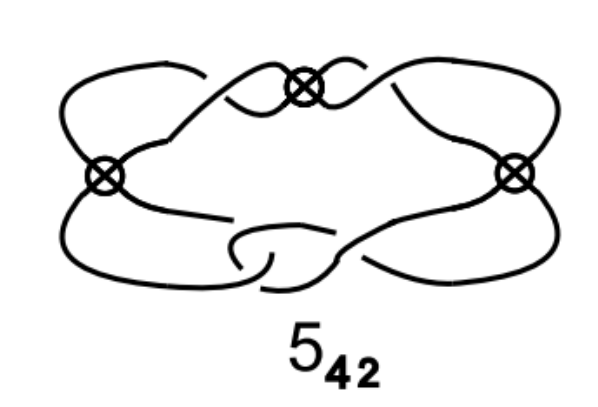} &
\includegraphics[height=1.5cm]{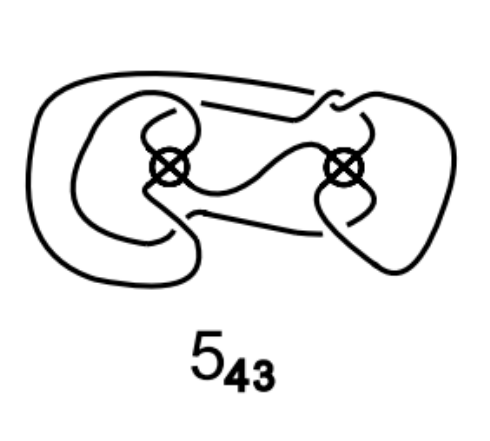} \\ \hline
\includegraphics[height=1.5cm]{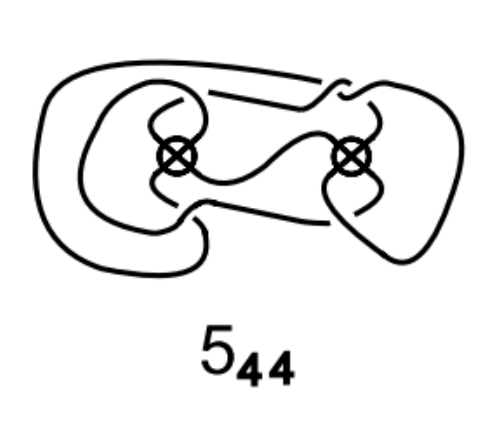} & 
\includegraphics[height=1.5cm]{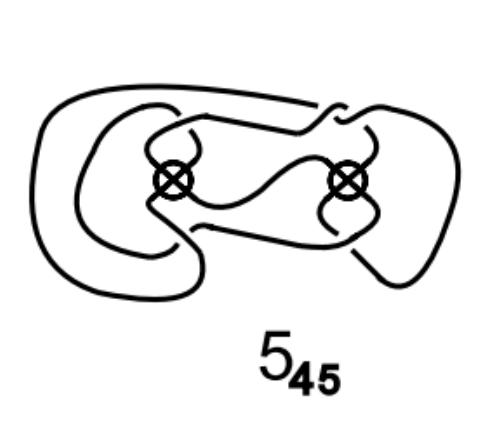} &
\includegraphics[height=1.5cm]{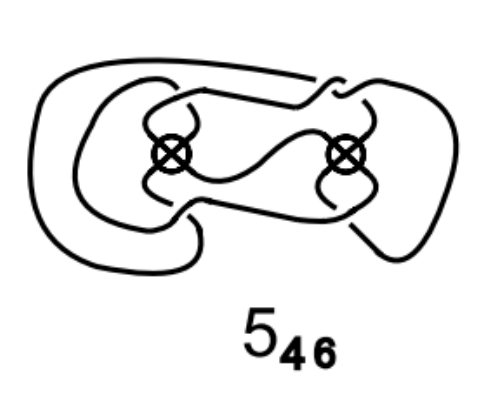} &
\includegraphics[height=1.5cm]{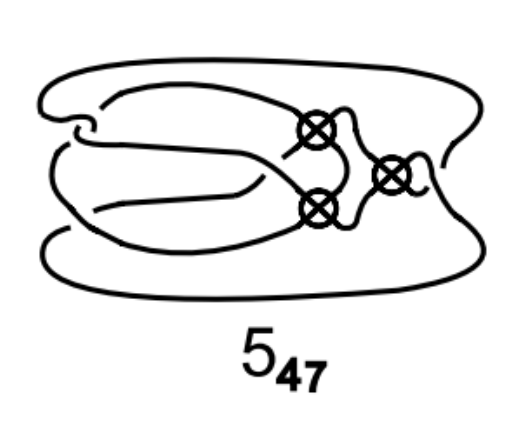} \\ \hline 
\includegraphics[height=1.5cm]{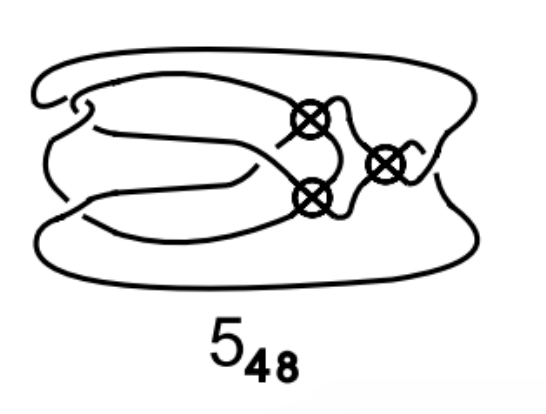} & 
\includegraphics[height=1.5cm]{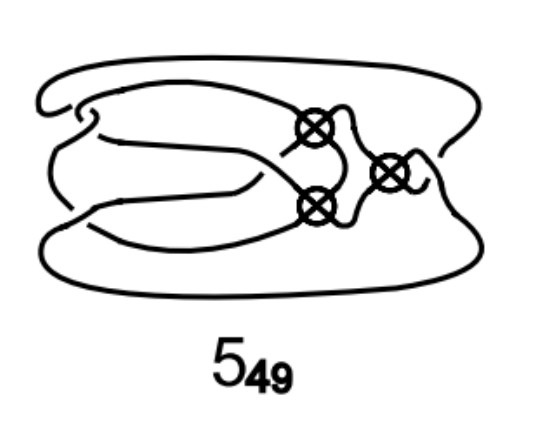} & 
\includegraphics[height=1.5cm]{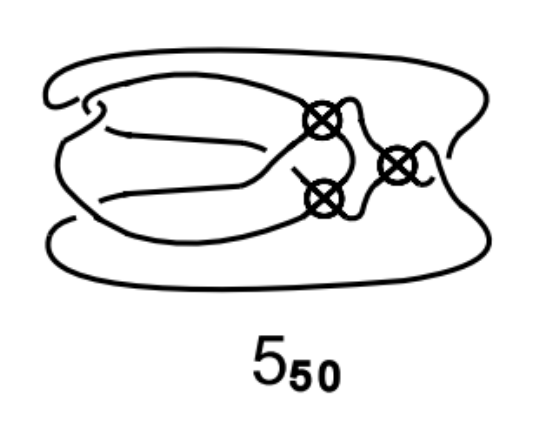} &
\includegraphics[height=1.5cm]{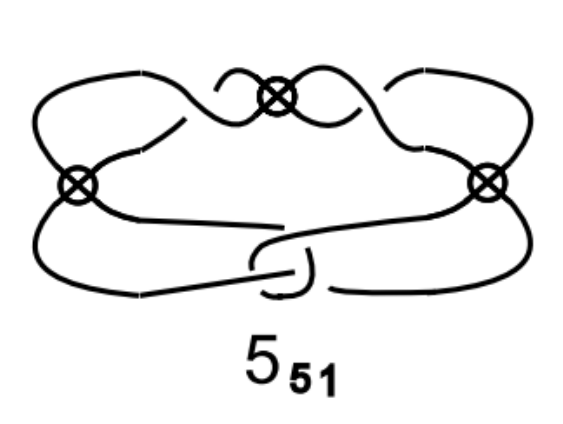} \\ \hline 
\includegraphics[height=1.5cm]{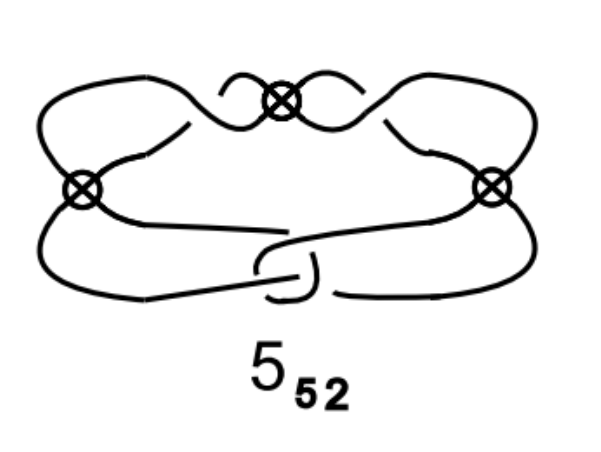} &  
\includegraphics[height=1.5cm]{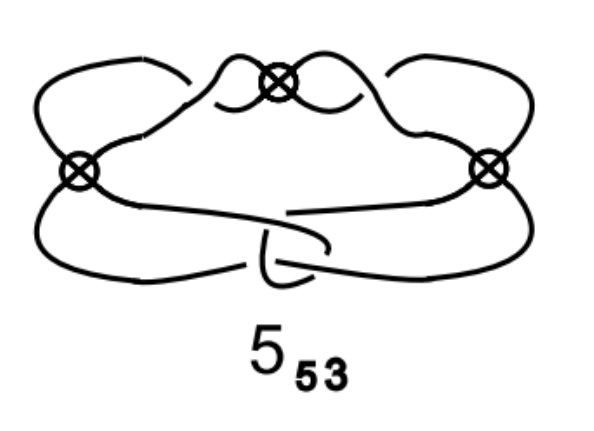} &
\includegraphics[height=1.5cm]{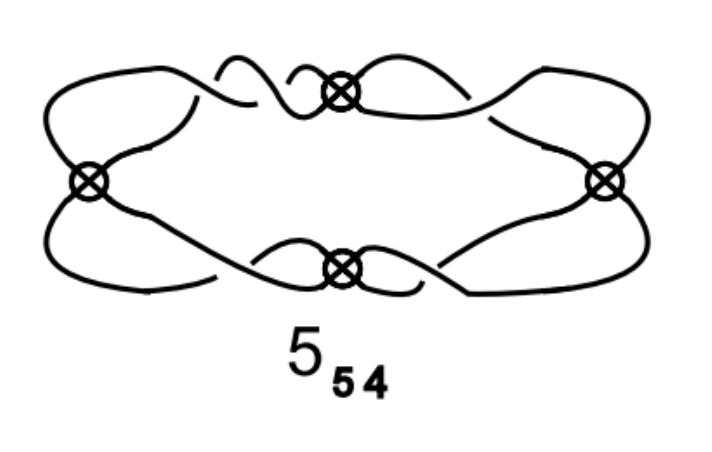} & 
\includegraphics[height=1.5cm]{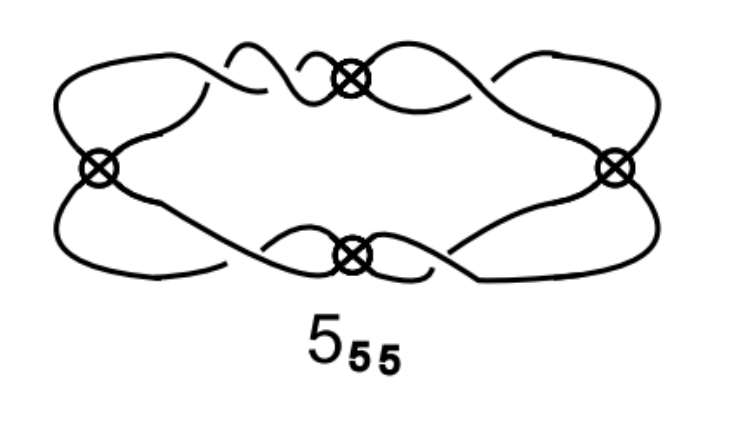} \\ \hline
\includegraphics[height=1.5cm]{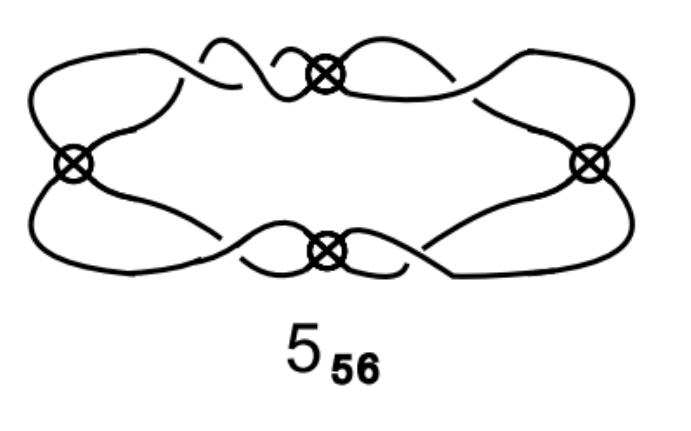} & 
\includegraphics[height=1.5cm]{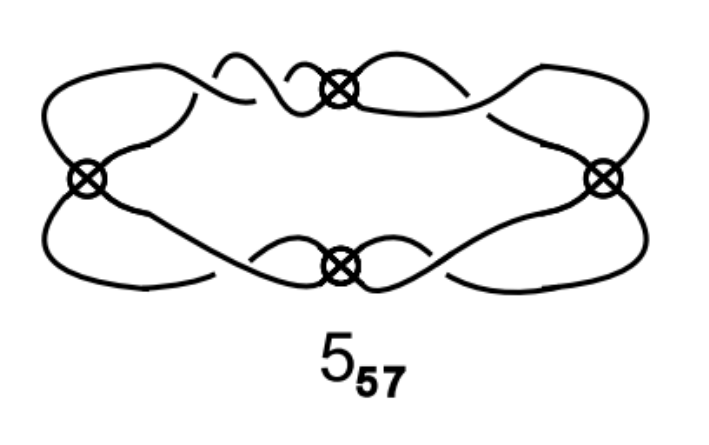} &
\includegraphics[height=1.5cm]{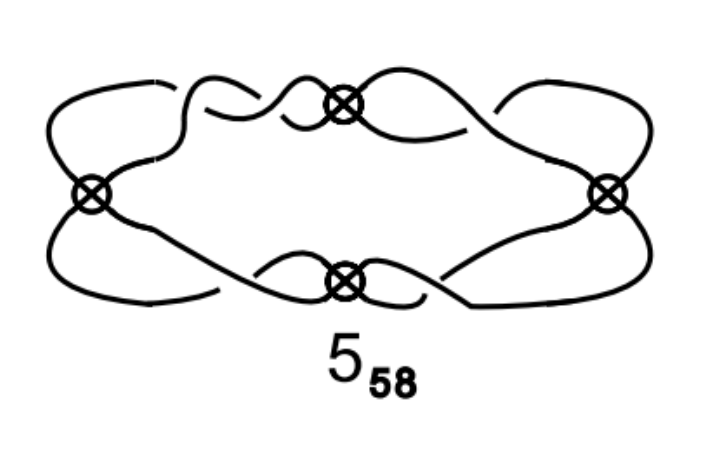} &
\includegraphics[height=1.5cm]{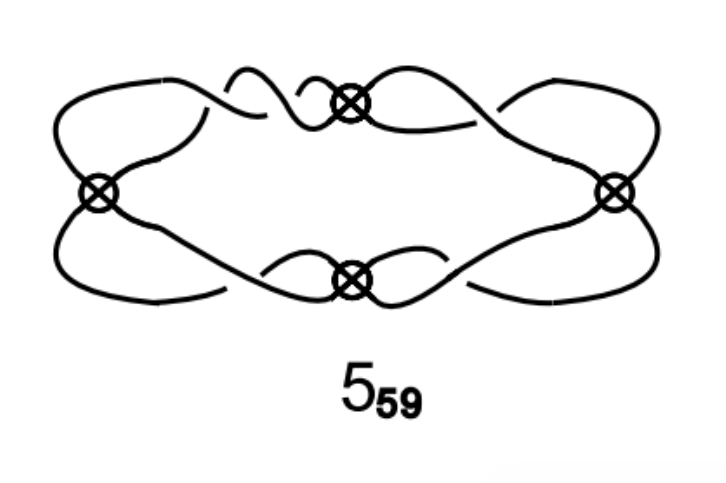} \\ \hline 
\includegraphics[height=1.5cm]{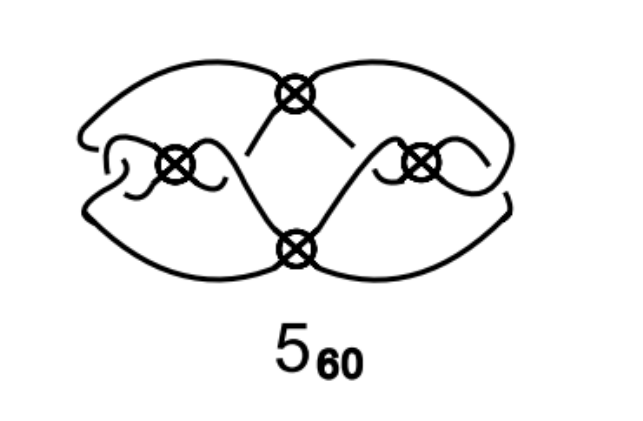} & 
\includegraphics[height=1.5cm]{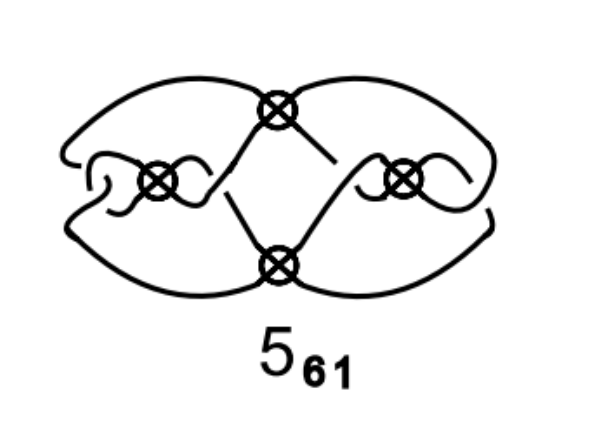} &
\includegraphics[height=1.5cm]{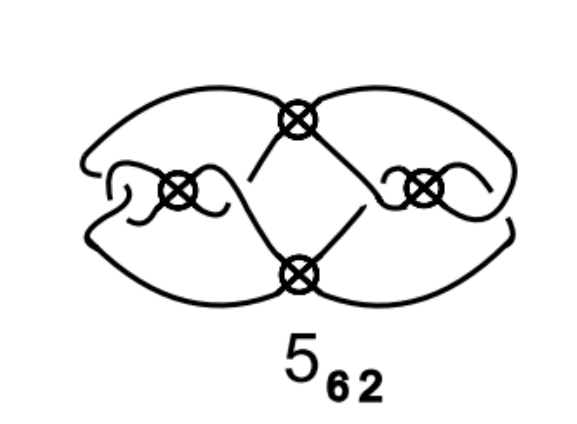} &
\includegraphics[height=1.5cm]{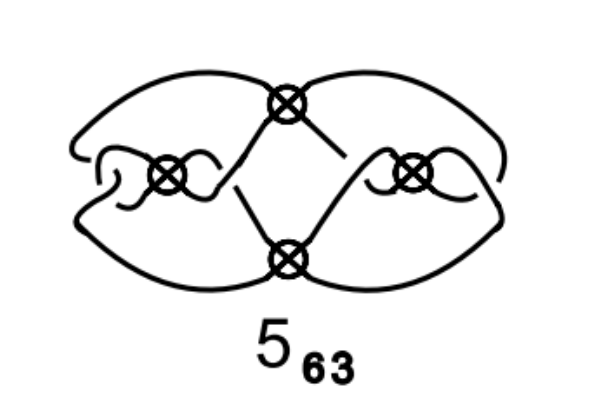} \\ \hline 
\includegraphics[height=1.5cm]{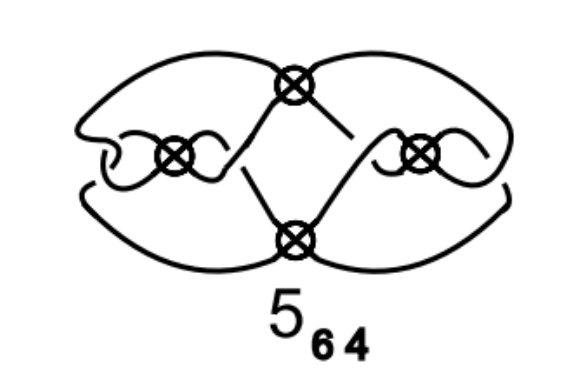} & 
\includegraphics[height=1.5cm]{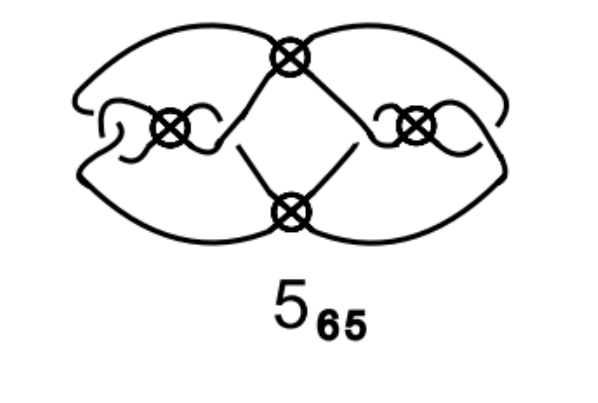} &
\includegraphics[height=1.5cm]{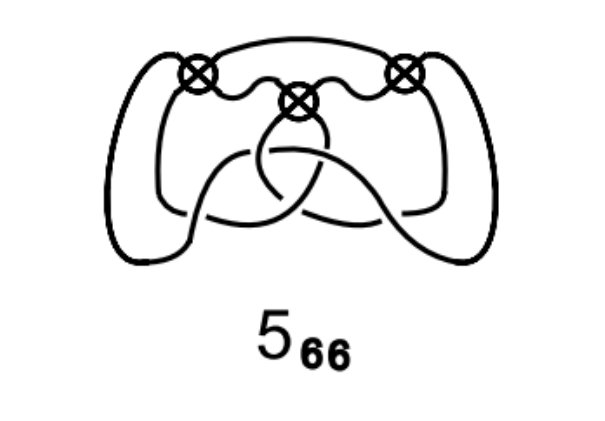} &
\includegraphics[height=1.5cm]{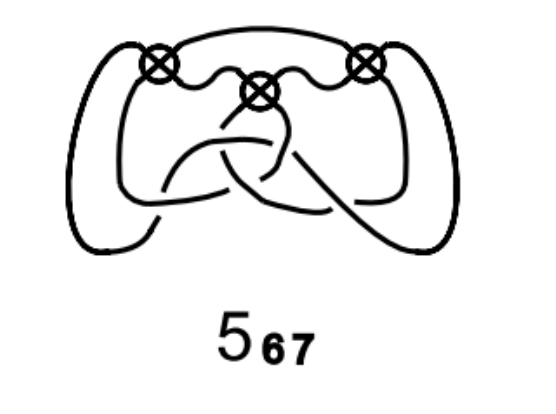} \\ \hline 
\includegraphics[height=1.5cm]{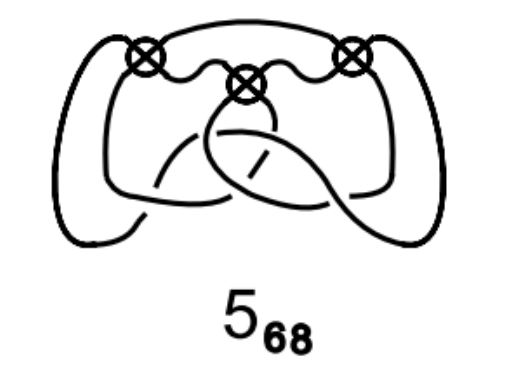} &
\includegraphics[height=1.5cm]{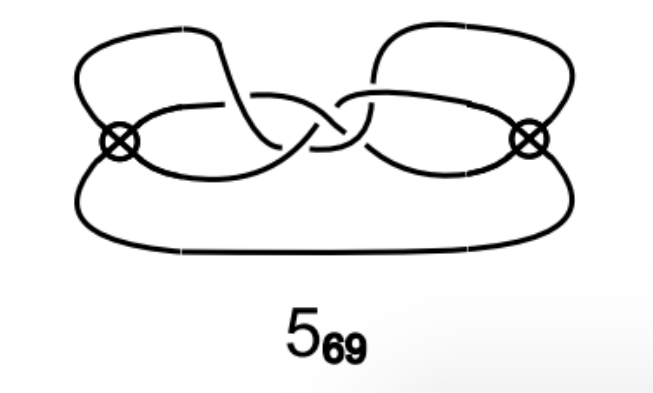} &  & \\ \hline
\end{tabular}
\end{table}

\end{document}